\documentclass[11pt,a4j]{article}
\usepackage{amsthm}

    \newtheorem{thm}{Theorem}[section]
    \newtheorem{prop}[thm]{Propositoin}
    \newtheorem{lem}[thm]{Lemma}
    \newtheorem{cor}[thm]{Corollary}
    
  \theoremstyle{definition}
    \newtheorem*{que}{Question}
    \newtheorem{defi}[thm]{Definition}
  \theoremstyle{remark}
    \newtheorem{rem}[thm]{Remark}
    \newtheorem{ex}[thm]{Example}

\usepackage{amsmath}
\usepackage{amssymb}
\usepackage[all]{xy}
\usepackage{graphicx}
\usepackage{mathrsfs}
\usepackage{indentfirst}
\usepackage{bm}

\title{Engel Manifolds and Contact 3-Orbifolds}
\author{Koji Yamazaki}

\begin{document}
\maketitle
\begin{abstract}
In early study of Engel manifolds from R. Montgomery, the Cartan prolongation and the development map are central figures. However, the development map can be globally defined only if the characteristic foliation is ``nice". In this paper, we introduce the Cartan prolongation of a contact 3-orbifold and the development map associated to a more general Engel manifold, and we give necessary and sufficient condition for the Cartan prolongation to be a manifold. Moreover, we define the Cartan prolongation of a $3$-dimensional contact \'{e}tale Lie groupoid and the development map associated to any Engel manifold. We prove that all Engel manifolds obtained as the Cartan prolongation of a ``space" with contact structure are obtained from a contact 3-orbifold.
\end{abstract}

\setcounter{section}{-1} 
	\section{Introduction} 

 A $2$-distribution $\mathcal{D}$ on a $4$-manifold $E$ is an Engel structure if $\mathcal{D}^2 := \mathcal{D} + [\mathcal{D} , \mathcal{D}]$ has rank $3$, and $\mathcal{D}^3 := \mathcal{D}^2 + [\mathcal{D}^2 , \mathcal{D}^2]$ has rank $4$. The pair $(E , \mathcal{D})$ is called an Engel manifold. Engel manifolds are very similar and closely related to contact manifolds (with a Legendrian foliation). For example, the projectivization $\mathbb{P} (\xi)$ of any contact structure $\xi$ on a contact $3$-manifold has an Engel structure. This is called the Cartan prolongation. On the other hand, any ``nice" Engel manifold has a bundle structure over a contact manifold. Any Engel manifold has a $1$-foliation $\mathcal{L}$, called the characteristic foliation, such that $[\mathcal{L} , \mathcal{D}^2] \subset \mathcal{D}^2$. The above word ``nice" means that the leaf space $E / \mathcal{L}$ is a manifold. Then $E / \mathcal{L}$ admits a contact structure $\xi$. Moreover, $E$ has a natural Engel morphism $\phi : E \rightarrow \mathbb{P} (\xi)$ called the development map. This discussion can be easily generalized to that for orbifolds. That is, if $E / \mathcal{L}$ is a orbifold, then $E / \mathcal{L}$ admits a contact structure, and then $E$ has the development map. However, the Cartan prolongation of a contact $3$-orbifold is an Engel orbifold in general.
\begin{que}
When is the Cartan prolongation of a contact $3$-orbifold  an Engel manifold?
\end{que}
 The following theorem gives a complete answer to this question.

\setcounter{section}{3}
\setcounter{thm}{0}
\begin{thm}
 Let $(\Sigma,\xi)$ be a contact 3-orbifold. The Cartan prolongation of $(\Sigma,\xi)$ is a manifold if and only if $(\Sigma,\xi)$ is positive, and $|G_x|$ is odd for all $x \in \Sigma$, where $G_x$ is the isotropy group at $x$.
\end{thm}

 Contact orbifolds are standard objects in various fields of mathmatics and physics. If a contact manifold has a ``symmetric property" (in mathmatics, that is an action of a Lie group), then we obtain the ``quotient" called the contact reduction \cite{geiges2008introduction}. Such an object is a contact orbifold. So we are rich in examples of contact orbifolds, and we can obtain more examples of Engel manifolds. This construction is functorial.\par
 In fact, an orbifold can be defined as a proper \'{e}tale Lie groupoid. In this context, the leaf space $E / \mathcal{L}$ can be regarded as an \'{e}tale Lie groupoid.

\setcounter{section}{4}
\setcounter{thm}{23}
\begin{cor}
 Let $\mathcal{G}$ be a $3$-dimensional contact \'{e}tale Lie groupoid. The Cartan prolongation of $\mathcal{G}$ is a manifold if and only if $\mathcal{G}$ is proper, positive, and $|\mathcal{G}_x|$ is odd for all objects $x \in \mathcal{G}_0$, where $\mathcal{G}_x$ is the isotropy group at $x$.
\end{cor}
 This corollary means that all Engel manifolds obtained as the Cartan prolongation of a ``space" with contact structure are obtained as above.\par
 Furthermore, we can define the development map associated to any Engel manifold as a ``map" (called generalized map) among Lie groupoids. This means that we can ``approximate" any Engel manifold to the Cartan prolongation by the universality of the development map.\par
 If you are interested in contact manifolds rather than Engel manifolds, you can replace 
Engel manifolds (resp. the characteristic foliations, contact manifolds and the Cartan prolongations)
with 
contact manifolds with a Legendrian foliation (resp. the fixed Legendrian foliations, manifolds and the spaces of contact elements) to develop the argument in parallel. For example, the contact version of the main result (Theorem \ref{result}) is correct. Moreover, we can define the development map for contact manifolds with a Legendrian foliation. And so on.

		\subsection*{Acknowledgements} 

 This article is established thanks to much support from my advisor Hisaaki Endo and helpful discussions with Jiro Adachi. I grateful and would like to thank them.

\setcounter{section}{0}
\setcounter{thm}{0}

	\section{Engel manifolds} \label{sec1}
Any Engel manifold has an even contact structure.
An even contact structure has the {\it characteristic foliation}.
The leaf space of the characteristic foliation has a contact structure.
Conversely, a contact $3$-manifold has an Engel manifold called the {\it Cartan prolongation}.
An Engel structure determines the {\it developing map}.
This section refers to \cite{montgomery1999engel}, \cite{adachi2002engel} and \cite{vogel2009existence}.
	\subsection{Contact manifolds and even contact manifolds}
An even contact manifold is an even-dimension version of a contact manifold.
An even contact manifold has an $1$-dimensional foliation called the {\it characteristic foliation}.
The leaf space of the characteristic foliation has a contact structure.
	\subsubsection{Contact manifolds}
\begin{defi}[contact manifolds]
Let $M$ be a $(2n+1)$-dimensional manifold.
A {\it contact structure} on the manifold $M$ is a corank $1$ distribution $\xi \subset TM$ such that a $(2n+1)$-form $(\wedge^n d \alpha) \wedge \alpha$ does not vanish anywhere for any local $1$-form $\alpha$ with $\xi = \operatorname{Ker}(\alpha)$.
Then, the pair $(M, \xi)$ is called a {\it contact manifold}.
And the above $1$-form $\alpha$ is called a {\it contact form}.
\end{defi}
We define a category of contact manifolds.
We have to define a morphism between contact manifolds.
In this paper, we use a {\it local contactomorphism} as a morphism.
\begin{defi}
A {\it local diffeomorphism} $f : M_1 \rightarrow M_2$ is a smooth map $f : M_1 \rightarrow M_2$ such that $df_x : T_xM_1 \rightarrow T_{f(x)}M_2$ is an isomorphism for any $x \in M_1$. \par
Let $(M_1, \xi_1)$ and $(M_2, \xi_2)$ be contact manifolds. 
A {\it local contactomorphism} $f : (M_1, \xi_1) \rightarrow (M_2, \xi_2)$ is a local diffeomorphism $f : M_1 \rightarrow M_2$ with $df(\xi_1) \subset \xi_2$. \par
A local contactomorphism $f : (M_1, \xi_1) \rightarrow (M_2, \xi_2)$ is a {\it contactomorphism} if $f : M_1 \rightarrow M_2$ is a bijection.
\end{defi}
\begin{prop} \label{contchara}
Let $M$ be a $(2n+1)$-dimensional manifold, and let $\alpha$ be a $1$-form.
Define $\xi = \operatorname{Ker}(\alpha)$.
Then, the followings are equivalent.
\begin{enumerate}
\renewcommand{\labelenumi}{(\arabic{enumi})}
\item The distribution $\xi$ is a contact structure.
\item The $1$-form $\alpha$ is a contact form.
\item A $2n$-form $(\wedge^n d \alpha) |_\xi$ on a distribution $\xi$ does not vanish anywhere.
\item A $2$-form $d \alpha |_\xi$ on a distribution $\xi$ is nondegenerate on $\xi$.
(i.e. A linear map $\xi \rightarrow \xi^\ast ; v \mapsto d \alpha (v, \cdot)$ is a bijection.)
\end{enumerate}
\end{prop}
\begin{proof}
${\it (1)} \Leftarrow {\it (2)} \Leftrightarrow {\it (3)}$ is obvious.
${\it (3)} \Leftrightarrow {\it (4)}$ is similar to the characterization of a symplectic form.
(See \cite{mcduff2017introduction}.)
${\it (1)} \Rightarrow {\it (2)}$ is easy.
(See \cite{geiges2008introduction}.)
\end{proof}
	\subsubsection{Even contact manifolds}
\begin{defi}[even contact manifolds]
Let $E$ be a $2n$-dimensional manifold.
An {\it even contact structure} on the manifold $E$ is a corank $1$ distribution $\mathcal{E} \subset TE$ such that a $2n$-form $(\wedge^{n-1} d \alpha) \wedge \alpha$ does not vanish anywhere for any local $1$-form $\alpha$ with $\mathcal{E} = \operatorname{Ker}(\alpha)$.
Then, the pair $(E, \mathcal{E})$ is called an {\it even contact manifold}.
And the above $1$-form $\alpha$ is called an {\it even contact form}.
\end{defi}
A morphism between even contact manifolds is defined in the same way as above.
\begin{defi}
Let $(E_1, \mathcal{E}_1)$ and $(E_2, \mathcal{E}_2)$ be even contact manifolds. 
A {\it local even contactomorphism} $f : (E_1, \mathcal{E}_1) \rightarrow (E_2, \mathcal{E}_2)$ is a local diffeomorphism $f : E_1 \rightarrow E_2$ with $df(\mathcal{E}_1) \subset \mathcal{E}_2$. \par
A local even contactomorphism $f : (E_1, \mathcal{E}_1) \rightarrow (E_2, \mathcal{E}_2)$ is an {\it even contactomorphism} if $f : E_1 \rightarrow E_2$ is a bijection.
\end{defi}
\begin{prop}
Let $E$ be a $2n$-dimensional manifold, and let $\alpha$ be a $1$-form.
Define a distribution $\mathcal{E}$ as $\mathcal{E} = \operatorname{Ker}(\alpha)$.
The followings are equivalent.
\begin{enumerate}
\renewcommand{\labelenumi}{(\arabic{enumi})}
\item The distribution $\mathcal{E}$ is an even contact structure.
\item The $1$-form $\alpha$ is an even contact form.
\item A $2n$-form $(\wedge^{n-1} d \alpha) |_\mathcal{E}$ on a distribution $\mathcal{E}$ does not vanish anywhere.
\item The rank of a linear map $\mathcal{E} \rightarrow \mathcal{E}^\ast ; v \mapsto d \alpha (v, \cdot)$ is $2n-2$.
\end{enumerate}
\end{prop}
\begin{proof}
It is similar to the proof of Proposition \ref{contchara}.
\end{proof}
Let $(E, \mathcal{E})$ be an even contact manifold.
Then, a distribution $\mathcal{L}( = \operatorname{Ker}[\mathcal{E} \rightarrow \mathcal{E}^\ast])$ has rank $1$.
Any $1$-dimensional distribution is integrable because $[X, X] = 0$.
The $1$-dimensional distribution $\mathcal{L}$ defnes a foliation.
(In this paper, we refer to an integrable distribution as a foliation.)
\begin{defi}[the characteristic foliation]
The above $1$-dimensional foliation $\mathcal{L}( = \operatorname{Ker}[\mathcal{E} \rightarrow \mathcal{E}^\ast])$ is called the {\it characteristic foliation} of the even contact manifold $(E, \mathcal{E})$.
\end{defi}
\begin{prop} \label{charfolcont}
Let $M \subset E$ be a codimension $1$ submanifold intersecting transversally with the characteristic foliation $\mathcal{L}$.
Then, a distribution $\xi = TM \cap \mathcal{E}$ on the submanifold $M$ is a contact structure.
\end{prop}
\begin{proof}
Consider the following diagram.
	\[\xymatrix{
		\xi \ar[d]_-{v \mapsto d \alpha (v,\cdot)} \ar@{^(->}[r]
		& \mathcal{E} \cong \xi \oplus \mathcal{L} \ar[d]^-{v \mapsto d \alpha (v,\cdot)}
		& \mathcal{L} \ar[d]^0 \ar@{^(->}[l]
	\\
		\xi^\ast
		& \mathcal{E}^\ast \cong \xi^\ast \oplus \mathcal{L}^* \ar@{->>}[l] \ar@{->>}[r]
		& \mathcal{L}^\ast.
	}\]
The rank of the kernel of the above map $\mathcal{L} \rightarrow \mathcal{L}^\ast$ is $1$.
The rank of the kernel of the above map $\mathcal{E} \rightarrow \mathcal{E}^\ast$ is $1$.
Then, the rank of the kernel of the above map $\xi \rightarrow \xi^\ast$ is $0$.
This means that the above map $\xi \rightarrow \xi^\ast$ is a bijection.
The distribution $\xi$ is a contact structure by Proposition \ref{contchara}.
\end{proof}
\begin{prop} \label{charfolhol}
Let $X$ be a smooth vector field tangent to the characteristic foliation $\mathcal{L}$.
Let $\rho_t$ be the flow of $X$.
Then, the (locally defined) diffeomorphism $\rho_t$ preserves the even contact structure $\mathcal{E}$.
\end{prop}
\begin{proof}
Let $\alpha$ be an even contact form of $\mathcal{E}$.
(i.e. $\mathcal{E} = \operatorname{Ker}(\alpha)$.)
We will show that there exists a time-dependent function $h_t$ satisfying the following equation.
\[
\rho_t^\ast \alpha = h_t \alpha.
\]
If we calculate the Lie derivative of both sides of the above equation, we get the following equation.
\[
\rho_t^\ast \mathcal{L}_X \alpha = \frac{d}{dt} \rho_t^\ast \alpha = \left(\frac{d}{dt} h_t\right) \alpha.
\]
(Where $\mathcal{L}_X(\cdot)$ is the Lie derivative by $X$.)
From the formula $\mathcal{L}_X = d i_X + i_X d$ on Lie derivatives, we get the following.
\[
\begin{array}{rcl}
\mathcal{L}_X \alpha & = & (d i_X + i_X d) \alpha \\
& = & d \alpha (X, \cdot) \in \operatorname{Ker}[T^\ast E \rightarrow \mathcal{E}^\ast].
\end{array}
\]
(Where $i_X$ is an interior product with the vector field $X$.) \par
Now, $\langle \alpha \rangle \subset \operatorname{Ker}[T^\ast E \rightarrow \mathcal{E}^\ast]$.
Furthermore, $\langle \alpha \rangle = \operatorname{Ker}[T^\ast E \rightarrow \mathcal{E}^\ast]$ by comparing these dimensions.
There exists a smooth function $f$ such that
\[
\mathcal{L}_X \alpha = f \alpha
\]
because $\mathcal{L}_X \alpha \in \langle \alpha \rangle$.
We obtain the following equation.
\[
\left(\frac{d}{dt} h_t\right) \alpha = (f \circ \rho_t) \rho_t^* \alpha = (f \circ \rho_t) h_t \alpha.
\]
Therefore, the differential equation to be satisfied by the unknown function $h_t$ is the following.
\[
\left\{
\begin{array}{lcl}
\frac{d}{dt} h_t & = & (f \circ \rho_t) h_t, \\
h_0 & = & 1.
\end{array}
\right.
\]
This equation can be solved concretely by $h_t = \exp(\int_0^t f \circ \rho_s ds)$.
\end{proof}
Proposition \ref{charfolcont} means that any section of the characteristic foliation has a contact structure.
Proposition \ref{charfolhol} means that any holonomy of the characteristic foliation preserves the contact structure.
Roughly speaking, these together mean that the leaf space of the characteristic foliation has a contact structure.
(See section \ref{sec1.3}.)
	\subsection{Engel manifolds} \label{sec1.2}
An Engel manifold has an even contact structure.
The leaf space of the characteristic foliation has a contact structure by the previous subsection.
Conversely, a contact $3$-manifold has an Engel manifold called the {\it Cartan prolongation}.
These two functors determine an adjunction. \par
Some concrete categories are represented by the followings.
\vspace{5pt}\par
${\bf Contact^3}$ is a category of contact $3$-manifolds. \par
{\bf Engel} is a category of Engel manifolds. \par
${\bf Engel^t}$ is a category of Engel manifolds with the trivial characteristic foliation.
(See Definition \ref{trichar}.)
\vspace{5pt}\par
First, we define a basic term.
\begin{defi}
Let $\mathcal{D}$ be a smooth distribution on a manifold.
(We can regard the distribution $\mathcal{D}$ as a sheaf of some vector fields.)
Define a sheaf $\mathcal{D}^n$ as a sheafification of a presheaf $\mathcal{D'}^n$ where the presheaf $\mathcal{D'}^n$ is inductively defined as follows.
\[
\begin{array}{lcl}
\mathcal{D'}^1 & = &\mathcal{D} \\
\mathcal{D'}^{n+1} & = & [\mathcal{D}^n, \mathcal{D}^n] + \mathcal{D}^n \\
& = & \{ [X, Y] + Z \, | \, X, Y, Z \in \mathcal{D}^n \}.
\end{array}
\]
\end{defi}
	\subsubsection{Definitions}
\begin{defi}[Engel manifolds]
Let $E$ be a $4$-dimensional manifold.
An {\it Engel structure} on $E$ is a $2$-dimensional distribution $\mathcal{D} \subset TE$ such that the sheaf $\mathcal{D}^2$ is $3$-dimensional distribution and the sheaf $\mathcal{D}^3$ is $4$-dimensional distribution.
Then, the pair $(E, \mathcal{D})$ is called an {\it Engel manifold}.
\end{defi}
A morphism between Engel manifolds is defined in the same way as the previous subsection.
\begin{defi}
Let $(E_1, \mathcal{D}_1)$ and $(E_2, \mathcal{D}_2)$ be Engel manifolds. 
A {\it local Engel diffeomorphism} $f : (E_1, \mathcal{D}_1) \rightarrow (E_2, \mathcal{D}_2)$ is a local diffeomorphism $f : E_1 \rightarrow E_2$ with $df(\mathcal{D}_1) \subset \mathcal{D}_2$. \par
A local Engel diffeomorphism $f : (E_1, \mathcal{D}_1) \rightarrow (E_2, \mathcal{D}_2)$ is an {\it Engel diffeomorphism} if $f : E_1 \rightarrow E_2$ is a bijection.
\end{defi}
Let $(E, \mathcal{D})$ be an Engel manifold.
Then, the distribution $\mathcal{E} = \mathcal{D}^2$ is an even contact structure on $E$.
The even contact manifold $(E, \mathcal{E})$ has the characteristic foliation $\mathcal{L}$.
The foliation $\mathcal{L}$ is called the characteristic foliation of $(E, \mathcal{D})$.
\begin{prop}
Let $(E, \mathcal{D})$ be an Engel manifold with the characteristic foliation $\mathcal{L}$.
Then, $\mathcal{L} \subset \mathcal{D}$.
\end{prop}
\begin{proof}
Let $\mathcal{E} = \mathcal{D}^2$.
In local, represent that $\mathcal{E} = \mathcal{D} \oplus V = \operatorname{Ker}(\alpha)$ ($^\exists V$ and $^\exists \alpha$).
Let $\mathcal{L'} = \operatorname{Ker}[\alpha^\dag]$ where
\[
\alpha^\dag : \mathcal{D} \rightarrow V^\ast ; X \mapsto d \alpha (X, \cdot).
\]
Then, $\mathcal{L'} \subset \mathcal{L} \cap \mathcal{D}$.
We will show that $\mathcal{L'} = \mathcal{L}$.
To show this, just show that the rank of the distribution $\mathcal{L'}$ is at least $1$. \par
There exist two smooth vector fields $X, Y \in \Gamma (\mathcal{E})$ such that a scalar function $\alpha ([X,Y])$ does not vanish anywhere bcause the even contact structure $\mathcal{E}$ has rank $3$ and the distribution $\mathcal{E}^2 (= \mathcal{D}^3)$ has rank $4$.
(Where $\Gamma(\cdot)$ is a sheaf of sections.)
We can represent that
\[
\left\{
\begin{array}{l}
X = X_1 + u \\
Y = Y_1 + v
\end{array}
\right.
\]
where $X_1, Y_1 \in \Gamma(\mathcal{D})$ and $u, v \in \Gamma(V)$.
From the formula $d\omega(Z, W) = Z(\omega(W)) - W(\omega(Z)) - \omega([Z, W])$, we get the following.
\[
\begin{array}{lcl}
d\alpha(X_1, v) & = & -\alpha ([X_1, v]) \\
d\alpha(Y_1, u) & = & -\alpha ([Y_1, u]).
\end{array}
\]
For any point $e \in E$, $\alpha_e ([X_1,v]_e)$ or $\alpha_e ([u,Y_1]_e)$ does not vanish because
\[
\begin{array}{lcl}
0 & \neq & \alpha ([X,Y]) \\
& = & \alpha ([X_1 + u,Y_1 + v]) \\
& = & \alpha ([X_1,v]) + \alpha ([u,Y_1]) \\
& = & - d\alpha(X_1, v) - d\alpha(Y_1, u).
\end{array}
\]
We obtain that there exist two tangent vectors $X_0 \in \mathcal{D}_e$ and $u_0 \in V_e$ for any point $e \in E$ such that $d\alpha_e(X_0, u_0) \neq 0$.
This means $\alpha^\dag (X_0) \neq 0$.
The rank of $\alpha^\dag$ is at least $1$.
The rank of $\mathcal{D}$ is $2$ and the rank of $V^\ast$ is $1$.
The rank of $\mathcal{L'} = \operatorname{Ker}[\alpha^\dag]$ is $1$.
Therefore, $\mathcal{L} = \mathcal{L'} \subset \mathcal{D}$.
\end{proof}
In order to describe the correspondence with contact manifolds, we define some useful words.
\begin{defi} \label{trichar}
Let $(E, \mathcal{D})$ be an Engel manifold.
We say that the Engel manifold $(E, \mathcal{D})$ has the {\it trivial characteristic foliation} if the leaf space of the characteristic foliation is a manifold.
${\bf Engel^t}$ is a category of Engel manifolds with the trivial characteristic foliation.
\end{defi}
\begin{prop} \label{leafsp}
Let $(E, \mathcal{D})$ be an Engel manifold with the trivial characteristic foliation $\mathcal{L}$.
Let $M$ be the leaf space of the foiation $\mathcal{L}$.
Then, a distribution $\xi = \mathcal{D}^2 / \mathcal{L}$ is well-defined on $M$, and the distribution $\xi$ is a contact structure on $M$.
\end{prop}
\begin{proof}
It is obvious by Proposition \ref{charfolcont} and Proposition \ref{charfolhol}.
\end{proof}
	\subsubsection{Cartan prolongations and developing maps}
Let $(M, \xi)$ be a contact 3-manifold.
Define a manifold $E$ as $E (= \mathbb{P} (\xi)) = \displaystyle{\coprod_{x \in M} } \mathbb{P} (\xi_x) \times \{ x \}$.
(Where $\mathbb{P}(V)$ is a projectization of a vector space $V$.
i.e. $\mathbb{P}(V) = \{ \mbox{dimension $1$ subspaces of $V$} \}$)
And let $\pi : E \rightarrow M$ be the projection. \par
We define a $2$-dimensional distribution $\mathcal{D}$ on the manifold $E$ in the following way. 
For each $(l, x) \in E$ with $\pi (l, x) = x$, $l \subset \xi_x$ is a dimension $1$ subspace of $\xi_x$.
Define a distribution $\mathcal{D}$ as $\mathcal{D}_{(l, x)} = d \pi_{(l, x)}^{-1} (l) \subset T_{(l, x)} E$. \\
Define a manifold $E'$ as $E' = \displaystyle{\coprod_{x \in M} } (\xi_x ? \mathbb{R}_{>0}) \times \{ x \}$.
Then, there exists a $2$-covering map $E' \rightarrow E$.
Let $\mathcal{D'}$ be a distribution on the manifold $E'$, defined as a pull-back of the distribution $\mathcal{D}$ by the covering map $E' \rightarrow E$.
\begin{lem} \label{Cartanlem}
\mbox{}
\begin{itemize}
\item The above distribution $\mathcal{D}$ is an Engel structure on the manifold $E$.
\item The distribution $\mathcal{D}^2$ coincides with the distribution $d\pi^{-1}(\xi)$.
\item The distribution $\operatorname{Ker}(d\pi)$ is the characteristic foliation of the Engel manifold $(E, \mathcal{D})$.
\end{itemize}
\end{lem}
\begin{proof}
In local, represent that $\xi = \langle X, Y \rangle$ where $X$ and $Y$ are smooth vector fields.
Then, the following map is a diffeomorphism.
\[
M \times S^1 \rightarrow E ; (x, [\theta]) \mapsto <\cos\left(\frac{\theta}{2}\right) X_x + \sin\left(\frac{\theta}{2}\right) Y_x>
\]
By this identification, the followings are vector fields on the manifold $E$.
\[
\begin{array}{lcl}
u & = & \cos\left(\frac{\theta}{2}\right) X + \sin\left(\frac{\theta}{2}\right) Y, \\
v & = & - \sin\left(\frac{\theta}{2}\right) X + \cos\left(\frac{\theta}{2}\right) Y
\end{array}
\]
(These are not well-defined in global, but these are well-defined in local.) \par
Then, 
\[
\begin{array}{lclclcl}
\mathcal{D} & = & \langle \partial_\theta, u \rangle, \\
\mathcal{D}^2 & = & \langle \partial_\theta, u, v \rangle & = & \langle \partial_\theta, X, Y \rangle & = & d \pi^{-1} (\xi), \\
\mathcal{D}^3 & = & d \pi^{-1} (TM) & = & TE.
\end{array}
\]
Therefore, the distribution $\mathcal{D}$ is an Engel structure on $E$.
Moreover, the distribution $\operatorname{Ker}(d \pi) (= \langle \partial_\theta \rangle)$ is the characteristic foliation.
\end{proof}
\begin{defi}[the Cartan prolongation]
The above Engel manifold $(E, \mathcal{D})$ (resp. $(E', \mathcal{D'})$) is called the {\it Cartan prolongation} (resp. {\it oriented Cartan prolongation}) of the contact manifold $(M, \xi)$.
This is denoted as $\mathbb{P}(M, \xi)$ (resp. $\mathbb{S}(M, \xi)$).
\end{defi}
The construction $(M, \xi) \mapsto \mathbb{P}(M, \xi)$ defines a functor $\mathbb{P} : {\bf Contact^3} \rightarrow {\bf Engel^t}$. \par
Let $\mathbb{L} : {\bf Engel^t} \rightarrow {\bf Contact^3}$ be a functor mapping each Engel manifold to the leaf space of the characteristic foliation.
Then, there exists an adjunction $\mathbb{L} \dashv \mathbb{P} : {\bf Engel^t} \rightarrow {\bf Contact^3}$.
The unit map $\phi : Id \rightarrow \mathbb{P} \circ \mathbb{L}$ of the adjunction $\mathbb{L} \dashv \mathbb{P}$ is called the {\it developing map}.
We will now construct the developing map in concrete.
\vspace{5pt}\par
Let $(E, \mathcal{D})$ be an Engel manifold with the trivial characteristic foliation $\mathcal{L}$.
Let $(M, \xi) = \mathbb{L}(E, \mathcal{D})$, and let $\pi : E \rightarrow M$ be the projection.
Define $\phi : E \rightarrow \mathbb{P}(\xi)$ as $\phi(e) = d\phi(\mathcal{D}_e)$ for $e \in E$.
\begin{lem} \label{lemdev}
The above map $\phi$ is a local Engel diffeomorphism.
\end{lem}
\begin{proof}
Take any $e \in E$, and take a foliated chart $(U; t, \bm{x})$ around $e$ with respect to the foliation $\mathcal{L}$.
(i.e. $U \cong \mathbb{R} \times \mathbb{R}^3 \ni (t, \bm{x})$ and $\mathcal{L} = \langle \partial_t \rangle$.)
Moreover, we can assume that $(\pi(U); \bm{x})$ is a chart of a manifold $M$ around $\pi(e) \in M$.
(Remark that we can assume $\pi(U) \cong \mathbb{R}^3$ because the holonomy group at the point $e$ is trivial.)
Take a vector field $u'$ on $U$ and vector fields $X$, $Y$ on $\pi(U)$ such that
\[
\begin{array}{lcl}
\mathcal{D} & = & \langle \partial_t, u' \rangle, \\
\xi & = & \langle X, Y \rangle.
\end{array}
\]
If necessary, we can assume that the vector field $u'$ is tangent to $T\mathbb{R}^3$ by replacing it with a linear transformation.
We can represent the vector field $u'$ as $u' = fX+gY$ where $f$ and $g$ are scalar functions.
We can assume that $\sqrt{f^2 + g^2} = 1$ by the normalization.
By the same identification $\mathbb{P}(\pi(U), \xi) \cong \mathbb{R}^3 \times S^1$ for a contact manifold $(\pi(U), \xi)$ as in Lemma \ref{Cartanlem}, $\phi : U \rightarrow \mathbb{P}(\pi(U), \xi)$ can be represented as follows.
\[
\phi (t, \bm{x}) = (\bm{x}, \theta (t, \bm{x})).
\]
Where $\theta : U \rightarrow S^1$ is a $S^1$-valued function such that $f = \cos(\theta /2) , g = \sin(\theta /2)$.
(Remark, $\theta$ has a lift function $U \rightarrow \mathbb{R}$ because $U$ is contractible.)
$(\partial_t f)X + (\partial_t g)Y (= [\partial_t ,u])$ does not vanish anywhere.
Then, $(\partial_t \theta)^2 (= 4(\partial_t f)^2 + 4(\partial_t g)^2)$ does not vanish anywhere.
This means that $\partial_t \theta$ does not vanish anywhere.
The map $\phi$ is a local diffeomorphism because $det(d \phi) (= \partial_t \theta)$ does not vanish anywhere. \par
$d \phi (\partial_t) (= d \theta (\partial_t))$ is tangent to $TS^1$.
In particular, this is included in the Engel structure $\langle \partial_\theta, u \rangle$ on the manifold $\mathbb{P}(\pi(U), \xi)$, which is the same description as Lemma \ref{Cartanlem}.
Moreover, $d \phi (u') = u+d \theta (u')$ is in the Engel structure $\langle \partial_\theta, u \rangle$.
Therefore, $\phi$ is a local Engel diffeomorphism.
\end{proof}
\begin{prop} \label{univdev}
Let $(N , \xi_N)$ be a contact $3$-manifold, and let $\psi : E \rightarrow \mathbb{P} (\xi_N)$ be an local Engel diffeomorphism.
Then, there exists an unique local contactomorphism $\tilde{\psi} : M \rightarrow N$ such that the following diagram is commutative.
	\[\xymatrix{
		E \ar[r]^-{\phi} \ar[rd]_-{\psi}
		& \mathbb{P} (M,\xi) \ar[d]^-{\mathbb{P} \tilde{\psi}}
	\\
		& \mathbb{P} (N,\xi_N)
	}\]
\end{prop}
\begin{proof}
Let $\tilde{\psi} (= \mathbb{L} \psi)$ be the induced map by the functor $\mathbb{L}$ from $\psi$.
For any $e \in E$,
\[
\begin{array}{lcl}
\mathbb{P} \tilde{\psi} \circ \phi (e) &=& d \tilde{\psi} \circ d \pi (\mathcal{D}_e) \\
&=& d \pi_N \circ d \psi (\mathcal{D}_e) \\
&=&  d \pi_N  ((\mathcal{D}_N)_{\psi (e)}) \\
&=& \psi (e)
\end{array}
\]
Conversely, suppose that there exists a $\tilde{\psi'}$ such that the above diagram is commutative.
Then, the following diagram is commutative.
	\[\xymatrix{
		(M, \xi) \ar@{=}[r] \ar[rd]_-{\tilde{\psi}}
		& (M, \xi) \ar[d]^-{\tilde{\psi'}}
	\\
		& (N, \xi_N)
	}\]
We obtain $\tilde{\psi'} = \tilde{\psi}$.
This means that such the local contactomorphism $\tilde{\psi}$ is unique.
\end{proof}
\begin{cor}
The functor $\mathbb{P} : {\bf Contact^3} \rightarrow {\bf Engel^t}$ is fully faithful.
\end{cor}
\begin{proof}
It is obvious by the counit of the adjunction $\mathbb{L} \dashv \mathbb{P}$ is an isomorphism.
(See \cite{mac2013categories}.)
\end{proof}
	\section{Contact Orbifolds} 

There exists a orbiifold structure on the leaf space of a foliation such that the quotient map is proper submersion if and only if all leaves are compact and the holonomy groupoid is locally finite.
(i.e. for any leaf $L \subset E$ and any object $x,y \in L$, the morphism set $Hol(E,\mathcal{L}) (x,y) = \{ hol(\gamma) | x = \gamma (0) , y = \gamma (1) \}$ is a finite set.)
In this section, we generalize the contents of section \ref{sec1}.
\begin{defi} \label{propercf}
Let $(E, \mathcal{D})$ be an Engel manifold.
We say that the Engel manifold $(E, \mathcal{D})$ has the {\it proper characteristic foliation} if the leaf space of the characteristic foliation is an orbifold.
${\bf Engel^p}$ is a category of Engel manifolds with the proper characteristic foliation.
\end{defi}
\begin{rem}
The above word {\it proper} means that the holonomy groupoid is {\it proper}. 
(cf. Definition \ref{defproper}.)
\end{rem}
Now, we define a contact orbifold.
\begin{defi}
Let $\Sigma$ be an orbifold, 
A {\it contact structure} on $\Sigma$ is a family of contact structures on each chart such that all local group actions are contact actions, and all transformation maps are contactomorphisms.
The pair $(\Sigma, \xi)$ of an orbifold $\Sigma$ and a contact structure $\xi$ on the orbifold $\Sigma$ is called a {\it contact orbifold}. 
(cf. Definition \ref{contLiegrpd})\\
Let $(\Sigma_1, \xi_1)$, $(\Sigma_2, \xi_2)$ be contact orbifolds.
A {\it local contactomorphism} $f : (\Sigma_1, \xi_1) \rightarrow (\Sigma_2, \xi_2)$ (cf. Definition \ref{contmor}) is a local diffeomorphism $f : \Sigma_1 \rightarrow \Sigma_2$ with the lifts are contact morphisms. \\
Similarly, we define an {\it Engel structure} on an orbifold, an {\it Engel orbifold}, and a {\it local Engel diffeomorphism}.
\end{defi}
In analogy with Proposition \ref{leafsp}, we just obtain the following proposition.
\begin{prop}
Let $(E, \mathcal{D})$ be an Engel manifold with the proper characteristic foliation $\mathcal{L}$.
The leaf space $\Sigma$ of the foiation $\mathcal{L}$ has a contact structure.
\end{prop}
Let $(\Sigma, \xi)$ be a contact 3-orbifold.
Then, $\mathbb{P} (\xi)$ and $\mathbb{S} (\xi)$ are orbifolds.
By Lemma \ref{Cartanlem}, these are Engel orbifolds.
\begin{defi}[the Cartan prolongation]
The above Engel orbifold $(E, \mathcal{D})$ (resp. $(E', \mathcal{D'})$) is called the {\it Cartan prolongation} (resp. {\it oriented Cartan prolongation}) of the contact orbifold $(\Sigma, \xi)$.
This is denoted as $\mathbb{P}(\Sigma, \xi)$ (resp. $\mathbb{S}(\Sigma, \xi)$).
\end{defi}
Let $(E , \mathcal{D})$ be an Engel manifold with the proper characteristic foliation, let $(\Sigma, \xi)$ be the leaf space of the characteristic foliation, and let $\pi : E \rightarrow \Sigma$ be the quotient map.
We define $\phi : E \rightarrow \mathbb{P} (\xi)$ as $\phi (e) = d \pi (\mathcal{D}_e) \subset \xi_{\pi (e)}$.
By Lemma \ref{lemdev} and Proposition \ref{univdev}, we obtain the following proposition immediately.
\begin{prop}
The above $\phi$ is a local Engel diffeomorphism $(E,\mathcal{D}) \rightarrow \mathbb{P} (\Sigma,\xi)$. 
Moreover, this satisfies the universality: 
For any contact 3-orbifold $(\Upsilon,\nu)$ and any Engel morphism $\psi : (E,\mathcal{D}) \rightarrow \mathbb{P} (\Upsilon,\nu)$, there exists a unique up to isomorphic contact morphism $\tilde{\psi} : (\Sigma,\xi) \rightarrow (\Upsilon,\nu)$ such that $\psi \cong \mathbb{P} (\tilde{\psi}) \circ \phi$.
\end{prop}

\begin{defi} 
The above $\phi$ is called the {\it development map} associated to $(E,\mathcal{D})$.
\end{defi}

At last, we describe {\it equivariant Darboux's Theorem}. \\
$\xi_{std} = Ker(dz+x dy - y dx)$ is a contact structure on $\mathbb{R}^3$.
We define $\phi_n , \psi : \mathbb{R}^3 \rightarrow \mathbb{R}^3$ as $\phi_n (x,y,z) = (x \cos(\frac{2 \pi}{n}) - y \sin(\frac{2 \pi}{n}) , x \sin(\frac{2 \pi}{n}) + y \cos(\frac{2 \pi}{n}) , z) , \psi (x,y,z) = (x,-y,-z)$.
$G_{n,std} = <\phi_n , \psi> \,^\curvearrowright (\mathbb{R}^3 , \xi_{std})$ and $H_{n,std} = <\phi_n> \,^\curvearrowright (\mathbb{R}^3 , \xi_{std})$ are contact actions of finite groups.
These are called {\it standard model}.
In fact, any contact 3-orbifold is locally isomorphic to a standard model. \\
Let $(\Sigma, \xi)$ be a contact 3-orbifold, and let $x \in \Sigma$.
We can take an orbifold chart $(V, G_x, p)$ around $x$.
In fact, we can take a metric on $V$ invariable for the action $G_x \,^\curvearrowright V$.
(For example, take a metric, and average it for the action $G_x \,^\curvearrowright V$.)
Then, $G_x \,^\curvearrowright T_x V$ preserves the metric.
Hence, $G_x \subset O(3)$.
$G_x \,^\curvearrowright T_x V$ preserves the rank $2$ subspace $\xi_x \subset T_x V$.
Then, $G_x = G_{n,std}$ or $G_x = H_{n,std}$.
By {\it Moser's trick}, we can prove the following theorem.
(See \cite{Daniel2013} or \cite{geiges2008introduction}.)
\begin{thm} \label{Darboux} 
Let $(\Sigma, \xi)$ be a contact 3-orbifold, and let $x \in \Sigma$.
There exists an orbifold chart $(V, G, p)$ around $x$ such that $(V, G)$ is isomorphic to an open neighborhood of $0 \in (\mathbb{R}^3,H)$ where $H = G_{n,std}$ or $H = H_{n,std}$.
\end{thm}
	\section{Main Result}
In general, the Cartan prolongation of a contact 3-orbifold is an Engel orbifold.
The main result in this paper is a necessary and sufficient condition for the Cartan prolongation to be a manifold.

\begin{thm} \label{result} 
Let $(\Sigma,\xi)$ be a contact 3-orbifold.,
\begin{enumerate}
\renewcommand{\labelenumi}{(\arabic{enumi})}
\item $\mathbb{P} (\Sigma,\xi)$ is a manifold if and only if $(\Sigma,\xi)$ is positive, and $|G_x|$ is odd for all $x \in \Sigma$.
\item $\mathbb{S} (\Sigma,\xi)$ is a manifold if and only if $(\Sigma,\xi)$ is positive.
\end{enumerate}
\end{thm}

\begin{defi} 
A contact orbifold $(\Sigma,\xi)$ is {\it positive} if all $G_x \,^\curvearrowright \xi_x$ preserve the orientation.
\end{defi}

\begin{ex}
The quotient space of $G_{n,std} \,^\curvearrowright \mathbb{R}^3$ is not positive. And the quotient space of $H_{n,std} \,^\curvearrowright \mathbb{R}^3$ is positive.
\end{ex}

Because of Theorem \ref{Darboux}, we have to only consider standard models. That is to say, Theorem \ref{result} follows the following lemma.

\begin{lem} \leavevmode 
\begin{enumerate}
\renewcommand{\labelenumi}{(\arabic{enumi})}
\item The actions $G_{n,std} \,^\curvearrowright \mathbb{P} (\xi_{std}) , G_{n,std} \,^\curvearrowright \mathbb{S} (\xi_{std})$ are not free.
\item The action $H_{n,std} \,^\curvearrowright \mathbb{P} (\xi_{std})$ is free if and only if $n$ is odd.
\item The action $H_{n,std} \,^\curvearrowright \mathbb{S} (\xi_{std})$ is free for any $n$.
\end{enumerate}
\end{lem}
\begin{proof}
 Recall $\xi_{std} = Ker(dz+x dy - y dx) , G_{n,std} = <\phi_n , \psi> , H_{n,std} = <\phi_n> , \phi_n (x,y,z) = (x \cos(\frac{2 \pi}{n}) - y \sin(\frac{2 \pi}{n}) , x \sin(\frac{2 \pi}{n}) + y \cos(\frac{2 \pi}{n}) , z)$, and $\psi (x,y,z) = (x,-y,-z)$. In particular, $\xi_{std,(0,0,z)} = Ker(dz) = <\partial_x , \partial_y>$.
\begin{description}
\item[{\it (1)}]
$d \psi (\partial_x) = \partial_x , \partial_x \in \xi_{std,0}$. The claim follows it immediately.
\item[{\it (2),(3)}]
 If $(x,y) \neq 0$, $H_{n,std} \,^\curvearrowright \mathbb{R}^3$ is free at $(x,y,z)$. So we have to only consider the case of $x=y=0$.\\ \hspace{5pt}
 Let $v_\theta = (\cos \theta) \partial_x + (\sin \theta) \partial_y$. $\mathbb{P} (\xi_{std,(0,0,z)})$ and $\mathbb{S} (\xi_{std,(0,0,z)})$ are generated by $v_\theta$. Moreover, $v_{\theta_1} \equiv v_{\theta_2}$ in $\mathbb{P} (\xi_{std,(0,0,z)})$ if and only if $\theta_1 \equiv \theta_2$ mod $\pi$. And $v_{\theta_1} \equiv v_{\theta_2}$ in $\mathbb{S} (\xi_{std,(0,0,z)})$ if and only if $\theta_1 \equiv \theta_2$ mod $2 \pi$.
\begin{eqnarray*}
 d \phi_n (v_\theta) &=& (\cos \theta \cos(\frac{2 \pi}{n}) - \sin \theta \sin(\frac{2 \pi}{n})) \partial_x + (\cos \theta \sin(\frac{2 \pi}{n}) + \sin \theta \cos(\frac{2 \pi}{n})) \partial_y \\
  &=& \cos (\theta + \frac{2 \pi}{n}) \partial_x + \sin (\theta + \frac{2 \pi}{n}) \partial_y \\
  &=& v_{\theta + \frac{2 \pi}{n}}.
\end{eqnarray*} \hspace{5pt}
Hence,
\begin{eqnarray*}
 d \phi_n^k (v_\theta) \equiv v_\theta \mbox{ in } \mathbb{P} (\xi_{std,(0,0,z)})
  & \Leftrightarrow & 2k \equiv 0 \mbox{ mod } n. \\
 d \phi_n^k (v_\theta) \equiv v_\theta \mbox{ in } \mathbb{S} (\xi_{std,(0,0,z)})
  & \Leftrightarrow & k \equiv 0 \mbox{ mod } n.
\end{eqnarray*} \hspace{5pt}
So, The action $H_{n,std} \,^\curvearrowright \mathbb{S} (\xi_{std})$ is free for any $n$. And, if $n$ is odd, then The action $H_{n,std} \,^\curvearrowright \mathbb{P} (\xi_{std})$ is free. Otherwise the action $H_{n,std} \,^\curvearrowright \mathbb{P} (\xi_{std})$ is not free
\end{description}
\end{proof}
	\section{Generalization to Lie groupoids}
Roughly speaking, a {\it lie groupoid} is a groupoid object in the category of manifolds and smooth maps.
For example, any holonomy groupoid is a lie groupoid.
In fact, an orbifold can be defined as a {\it proper, \'{e}tale lie groupoid}.
See Appendix \ref{apx1}. \\
In this section, we prove that the Cartan prolongation of a contact \'{e}tale lie groupoid $\mathcal{G}$ is a manifold only if $\mathcal{G}$ is proper, and we define the {\it development map} associated to any Engel manifold.
\begin{defi} \label{contLiegrpd}
A {\it contact \'{e}tale Lie groupoid} (resp. An {\it Engel \'{e}tale Lie groupoid}) $\mathcal{G}$ is an \'{e}tale Lie groupoid $\mathcal{G}$ 
(cf. Deffinition \ref{Liegrpd} and Definition \ref{etale})
such that 
\begin{itemize}
\item The manifolds $\mathcal{G}_0$ and $\mathcal{G}_1$ are contact manifolds (resp. Engel manifolds).
\item The maps $s$, $t$, $i$, $\operatorname{inv}$ and $\operatorname{comp}$ are local cotactomorphisms (resp. local Engel diffeomorphisms).
\end{itemize}
A {\it contact orbifold} (resp. An {\it Engel orbifold}) is a proper contact \'{e}tale Lie groupoid (resp. a proper Engel \'{e}tale Lie groupoid).
\end{defi}
\begin{ex}
Let $(E, \mathcal{D})$ be an Engel manifold.
Define a \'{e}tale Lie groupoid $\mathbb{L}(E, \mathcal{D})$ as the maximal holonomy groupoid of the characteristic foliation of $(E, \mathcal{D})$.
(cf. Example \ref{leafspex})
This is a contact \'{e}tale Lie groupoid by Proposition \ref{charfolcont} and Proposition \ref{charfolhol}.
\end{ex}
We define a morphism.
\begin{defi} \label{contmor}
Let $\mathcal{G}$, $\mathcal{H}$ be contact \'{e}tale Lie groupoids (resp. Engel \'{e}tale Lie groupoids). \\
A {\it strict local contactomorphism} (resp. A {\it strict local Engel diffeomorphism}) $f : \mathcal{G} \rightarrow \mathcal{H}$ is a strict morphism $f : \mathcal{G} \rightarrow \mathcal{H}$ such that the map of objects $f_0 : \mathcal{G}_0 \rightarrow \mathcal{H}_0$ and the map of morphisms $f_1 : \mathcal{G}_1 \rightarrow \mathcal{H}_1$ are local contactomorphisms (resp. local Engel diffeomorphisms). \\
A {\it local contactomorphism} (resp. a {\it local Engel diffeomorphism}) $f : \mathcal{G} \rightarrow \mathcal{H}$ is a generalized map $f (= (e_f, f')) : \mathcal{G} \leftarrow \mathcal{G'} \rightarrow \mathcal{H}$ (cf. Definition \ref{genmap}) such that the Lie groupoid $\mathcal{G'}$ is a contact \'{e}tale Lie groupoid (resp. Engel \'{e}tale Lie groupoid) and that the strict morphisms $e_f$ and $f'$ are strict local contactomorphisms (resp. strict local Engel diffeomorphisms).
\end{defi}
\vspace{5pt} \par
Let $\mathcal{G}$ be a contact \'{e}tale Lie groupoid.
Then, an Engel \'{e}tale Lie groupoid $\mathbb{P}\mathcal{G}$ is determined as follows.
\[
\begin{array}{lcl}
(\mathbb{P}\mathcal{G})_0 & = & \mathbb{P}(\mathcal{G}_0), \\
(\mathbb{P}\mathcal{G})_1 & = & \mathbb{P}(\mathcal{G}_1).
\end{array}
\]
The Engel \'{e}tale Lie groupoid $\mathbb{P}\mathcal{G}$ is called the {\it Cartan prolongation} of the contact \'{e}tale Lie groupoid $\mathcal{G}$. \\
Similarly, we define an Engel \'{e}tale Lie groupoid $\mathbb{S} (\mathcal{G})$.
The Engel \'{e}tale Lie groupoid $\mathbb{S}\mathcal{G}$ is called the {\it oriented Cartan prolongation} of the contact \'{e}tale Lie groupoid $\mathcal{G}$.
\begin{prop}
The Cartan prolongation $\mathbb{P}\mathcal{G}$ (resp. The oriented Cartan prolongation $\mathbb{S}\mathcal{G}$) is an orbifold if and only if the contact \'{e}tale Lie groupoid $\mathcal{G}$ is an orbifold.
\end{prop}
\begin{proof}
The Engel \'{e}tale Lie groupoid $\mathbb{P}\mathcal{G}$ is proper if and only if the contact \'{e}tale Lie groupoid $\mathcal{G}$ is proper, because of the following diagram.
	\[\xymatrix{
		(\mathbb{P}\mathcal{G})_1 \ar[r]^-{s \times t} \ar[d]_-{proper}
		& (\mathbb{P}\mathcal{G})_0 \times (\mathbb{P}\mathcal{G})_0 \ar[d]^-{proper}
	\\
		 \mathcal{G}_1 \ar[r]^-{s \times t}
		&  \mathcal{G}_0 \times \mathcal{G}_0.
	}\]
Therefore, the Engel \'{e}tale Lie groupoid $\mathbb{P}\mathcal{G}$ is an orbifold if and only if the contact \'{e}tale Lie groupoid $\mathcal{G}$ is an orbifold.
\end{proof}
From Theorem \ref{result}, we obtain the following result.
\begin{cor} \label{result2}
Let $\mathcal{G}$ be a $3$-dimensional contact \'{e}tale Lie groupoid.
\begin{enumerate}
\renewcommand{\labelenumi}{(\arabic{enumi})}
\item $\mathbb{P} (\mathcal{G})$ is a manifold if and only if $\mathcal{G}$ is proper, positive, and $|\mathcal{G}_x|$ is odd for all $x \in \mathcal{G}_0$.
\item $\mathbb{S} (\mathcal{G})$ is a manifold if and only if $\mathcal{G}$ is proper and positive.
\end{enumerate}
\end{cor}
At last, we define the development map.\par
Let $(E, \mathcal{D})$ be an Engel manifold with the characteristic foliation $\mathcal{L}$, and let $S \subset E$ be a global section of $\mathcal{L}$. 
Take the weak pull-back $\mathcal{P}$ of $E \rightarrow Hol(E,\mathcal{L})$ and $Hol_S(E, \mathcal{L}) \rightarrow Hol(E,\mathcal{L})$.
(cf. Definition \ref{wpl})
	\[\xymatrix{
		\mathcal{P} \ar[r]^-{\pi} \ar[d]_-{e_\phi}
		& Hol_S(E, \mathcal{L}) \ar[d]
	\\
		E \ar[r]
		& Hol(E, \mathcal{L})
	}\] \par 
$e_\phi : \mathcal{P} \rightarrow E$ is a local diffeomorphism. 
Hence, the Engel structure on $\mathcal{P}$ is induced by $e_\phi$.
Moreover, $(Hol_S(E,\mathcal{L}))_0$ (resp. $(Hol_S(E,\mathcal{L}))_1$) is the leaf space of $\mathcal{P}_0$ (resp. $\mathcal{P}_1$), and $\pi_0$ (resp. $\pi_1$) is the quotient map.
When we denote $\mathcal{P} = (\mathcal{P}_0 , \mathcal{P}_1 , s,t,i,inv,comp)$, we obtain \\
$\mathbb{L}(E, \mathcal{L})_0 = (\mathbb{L} (\mathcal{P}_0) , \mathbb{L} (\mathcal{P}_1) , \mathbb{L} (s) , \mathbb{L} (t) , \mathbb{L} (i) , \mathbb{L} (inv) , \mathbb{L} (comp))$ where $\mathbb{L}$ is the functor defined in Proposition \ref{leafsp}.

Now, we define $\phi' : \mathcal{P} \rightarrow \mathbb{P} (Hol_S(E,\mathcal{L}))$ as $\phi'_0 , \phi'_1$ is the development map associated to $\mathcal{P}_0 , \mathcal{P}_1$, respectively. $\phi = (e_\phi, \mathcal{P} , \phi')$ is a generalized map $\phi : E \dashrightarrow \mathbb{P} (Hol_S(E,\mathcal{L}))$.

\begin{prop} 
 The above $\phi'$ is a strict Engel morphism $\mathcal{P} \rightarrow \mathbb{P} (Hol_S(E,\mathcal{L}))$. That is to say, The above $\phi$ is a (generalized) Engel morphism $(E , \mathcal{D} ) \dashrightarrow \mathbb{P} (Hol_S(E,\mathcal{L}))$. Moreover, this satisfies the universality: For any $3$-dimensional contact \'{e}tale Lie groupoid $\mathcal{G}$ and any Engel morphism $\psi : (E,\mathcal{D}) \dashrightarrow \mathbb{P} (\mathcal{G})$, there exists a unique up to isomorphic contact morphism $\tilde{\psi} : Hol_S(E,\mathcal{L}) \rightarrow \mathcal{G}$ such that $\psi \cong \mathbb{P} (\tilde{\psi}) \circ \phi$.
\end{prop}

\begin{defi} 
 The above $\phi$ is called the {\it development map} associated to $(E,\mathcal{D})$.
\end{defi}
\begin{rem}
 $e_\phi$ is essentially surjective, and $E$ is a manifold. So, $(e_\phi)_0$ is surjective.
\end{rem}

 This means that any Engel manifold is locally isomorphic to the Cartan prolongation of a contact manifold. From {\it Darboux's Theorem} for contact manifolds (and the lifting $S^1 \, ^\curvearrowright \mathbb{P} (\mathbb{R}^3 , \xi
_{std})$ of the rotary action $S^1 \, ^\curvearrowright (\mathbb{R}^3 , \xi_{std})$), we obtain the Engel version of that. (This is well-known.)

\begin{cor}
 Let $(E , \mathcal{D} )$ be an Engel manifold, and let $p \in E$. There exists a chart $(U;x,y,z,w)$ with $p \in U$ such that $\mathcal{D} |_U = Ker(dy-zdx) \cap Ker(dz-wdx)$.
\end{cor}
	\appendix
	\section{Lie groupoids and orbifolds} \label{apx1}
We define an orbifold as a Lie groupoid.
(See \cite{moerdijk2003introduction}.)
	\subsection{Lie groupoids}
A {\it groupoid} is a category such that all morphisms are isomorphisms.
A groupoid consists of seven data $\mathcal{G} = (\mathcal{G}_0, \mathcal{G}_1, s, t, i, \operatorname{inv}, \operatorname{comp})$, two sets $\mathcal{G}_0$ and $\mathcal{G}_1$ and five maps $s$, $t$, $i$, $\operatorname{inv}$ and $\operatorname{comp}$.
Each data has a name as follows.
\begin{itemize}
\item $\mathcal{G}_0$ is the set of objects.
\item $\mathcal{G}_1$ is the set of arrows.
\item $s : \mathcal{G}_1 \rightarrow \mathcal{G}_0$ is the souce map $(\sigma : x \rightarrow y) \mapsto x$.
\item $t : \mathcal{G}_1 \rightarrow \mathcal{G}_0$ is the target map $(\sigma : x \rightarrow y) \mapsto y$.
\item $i : \mathcal{G}_0 \rightarrow \mathcal{G}_1$ is the identities map.
\item $\operatorname{inv} : \mathcal{G}_1 \rightarrow \mathcal{G}_1$ is the inversion map.
\item $\operatorname{comp} : \mathcal{G}_2 \rightarrow \mathcal{G}_1$ is the composition map.
\end{itemize}
Where, $\mathcal{G}_2$ is the set of composable pairs of arrows.
(In general, $\mathcal{G}_n$ is the set of composable $n$-pairs of arrows.) \par
Roughly speaking, a {\it Lie groupoid} is a groupoid object in the category of manifolds.
However, for technical reasons, we will assume that the souce map is a submersion.
\begin{defi}[Lie groupoids] \label{Liegrpd}
A {\it Lie groupoid} is a groupoid $\mathcal{G} = (\mathcal{G}_0, \mathcal{G}_1, s, t, i, \operatorname{inv}, \operatorname{comp})$ such that 
\begin{itemize}
\item $\mathcal{G}_0$ and $\mathcal{G}_1$ are manifolds.
\item $s$, $t$, $i$, $\operatorname{inv}$ and $\operatorname{comp}$ are smooth.
\item $s$ is a submersion.
\end{itemize}
\end{defi}
\begin{ex} \label{actlie}
Let $M$ be a manifold, and let $G$ be Lie group.
Suppose that $G$ acts on $M$.
The {\it action groupoid} associated with the action $G ^\curvearrowright M$ consists the following data.
\begin{itemize}
\item $(G \ltimes M)_0 = M$.
\item $(G \ltimes M)_1 = G \times M$.
\item $s(\sigma, x) = x$.
\item $t(\sigma, x) = \sigma \cdot x$.
\end{itemize}
\end{ex}
\begin{ex}[holonomy groupoids]
Let $E$ be a $n$-manifold, and let $\mathcal{F}$ be a $p$-dimensional foliation on $E$.
The holonomy groupoid $Hol(E, \mathcal{F})$ of $(E, \mathcal{F})$ consists the following data.
\begin{itemize}
\item $Hol(E, \mathcal{F})_0 = E$.
\item An arrow $\sigma : x \rightarrow y$ ($x, y \in E$) is a holonomy from $x$ to $y$ of the foliation $\mathcal{F}$.
\end{itemize}
Then $Hol(E, \mathcal{F})$ is a Lie groupoid.
See \cite{moerdijk2003introduction}.
\end{ex}
A (strict) morphism between Lie groupoids is a functor, which is composed of smooth maps.
Moreover, a (strict) $2$-morphism is a natural transformation, which is composed of smooth maps.
These determine a $2$-category {\bf LieGrpd} of Lie groupoids.
\begin{defi} \label{etale}
A Lie groupoid $\mathcal{G}$ is {\it \'{e}tale} if the souce map $s$ is a local diffeomorphism.
\end{defi}
\begin{ex} \label{actdiscrete}
Let $M$ be a manifold, and let $G$ be Lie group.
Suppose that $G$ acts on $M$.
If $G$ is discrete, then the action groupoid associated with the action $G ^\curvearrowright M$ is \'{e}tale.
(cf. Example \ref{actlie})
\end{ex}
\begin{ex}[\v{C}ech groupoids]
Let $M$ be a manifold, and let $\mathcal{U} (= \{ U_\lambda \})$ be an open covering of the manifold $M$.
The {\it \v{C}ech groupoid} $\mathcal{G}$ of  the open covering $\mathcal{U}$ consists the following data.
\begin{itemize}
\item $\mathcal{G}_0 = \coprod_\lambda U_\lambda$.
\item $\mathcal{G}_1 = \coprod_{\lambda, \mu} U_{\lambda\mu}$.
\item $s : \coprod_{\lambda, \mu} U_{\lambda\mu} \rightarrow \coprod_\lambda U_\lambda$.
\item $t : \coprod_{\lambda, \mu} U_{\lambda\mu} \rightarrow \coprod_\mu U_\mu$.
\end{itemize}
Then, $\mathcal{G}$ is an \'{e}tale Lie groupoid. \par
In particular, if $\mathcal{U} = \{ M \}$, then the \v{C}ech groupoid is called the {\it trivial groupoid} of $M$.
\end{ex}
\begin{ex}[holonomy groupoids on sections]
Let $E$ be a $n$-manifold, and let $\mathcal{F}$ be a $p$-dimensional foliation on $E$.
A {\it section} of the foliation $\mathcal{F}$ is a $q$-manifold ($q = n-p$) immersed in $E$ intersecting transversally with the foliation $\mathcal{F}$.
A {\it global section} of the foliation $\mathcal{F}$ is a section of the foliation $\mathcal{F}$ intersecting transversally with all leaves.
The {\it maximal section} of the foliation $\mathcal{F}$ is a global section $\coprod_{S \in \mathcal{S}} S$ where $\mathcal{S}$ is a set of all sections of the foliation $\mathcal{F}$. \par
Let $S \subset E$ be a global section of the foliation $\mathcal{F}$.
The holonomy groupoid $Hol_S(E, \mathcal{F})$ on $S$ consists the following data.
\begin{itemize}
\item $Hol_S(E, \mathcal{F})_0 = S$.
\item An arrow $\sigma : x \rightarrow y$ ($x, y \in S$) is a holonomy from $x$ to $y$ of the foliation $\mathcal{F}$.
\end{itemize}
Then $Hol_S(E, \mathcal{F})$ is an \'{e}tale Lie groupoid. \par
If $S$ is the maximal section, then we call $Hol_S(E, \mathcal{F})$ the {\it maximal holonomy groupoid} of the foliation $\mathcal{F}$.
\end{ex}
\begin{rem}
In fact, the above class $\mathcal{S}$ is not a set.
Therefore, strictly speaking, the definition of a {\it section} needs to be modified as follows. \par
A {\it section} of the foliation $\mathcal{F}$ is a immersion $S \rightarrow E$ intersecting transversally with the foliation $\mathcal{F}$ such that $S$ is a $q$-submanifold $S \subset \mathbb{R}^k$ ($q = n-p$) for a sufficiently large integer $k$.
\end{rem}
We want to regard Lie groupoids as generalized spaces.
We have to regard certain morphisms as isomorphisms, for example, refinements of \v{C}ech groupoids.
Then, a {\it generalized morphism} is defined by a right fraction.
\begin{defi}[weak equivalences]
Let $\mathcal{G}$ and $\mathcal{H}$ be Lie groupoids.
A morphism $f : \mathcal{G} \rightarrow \mathcal{H}$ is a {\it weak equivalence} if it saticefies the following condition.
\begin{description}
\item[(essentially surjective)]
\mbox{}\\
$t \circ pr_1 : \mathcal{H}_1 \times_{\mathcal{H}_0} \mathcal{G}_0 \rightarrow \mathcal{H}_0$ is a surjective submersion, where 
\[
\mathcal{H}_1 \times_{\mathcal{H}_0} \mathcal{G}_0 = \{ (\sigma, x) \in \mathcal{H}_1 \times \mathcal{G}_0 \, | \, s(\sigma) = f(x) \}.
\]
\item[(fully faithful)]
\mbox{}\\
The following diagram is a pull-back.
	\[\xymatrix{
		\mathcal{G}_1 \ar[r]^-{f} \ar[d]_-{s \times t}
		& \mathcal{H}_1 \ar[d]^-{s \times t}
	\\
		 \mathcal{G}_0 \times \mathcal{G}_0 \ar[r]^-{f \times f}
		&  \mathcal{H}_0 \times \mathcal{H}_0
	}\]
\end{description}
\end{defi}
\begin{ex}[refinements]
Let $M$ be a manifold, and let $\mathcal{U}$ and $\mathcal{V}$ be open coverings of $M$.
Suppose that $\mathcal{V}$ is a refinment of $\mathcal{U}$.
Let $\mathcal{G}$ (resp. $\mathcal{H}$) be the \v{C}ech groupoid of $\mathcal{U}$ (resp. $\mathcal{V}$).
Then, the inclusion morphism $\mathcal{H} \rightarrow \mathcal{G}$ is a weak equivalence.
\end{ex}
\begin{defi}[generalized morphisms] \label{genmap}
Let $\mathcal{G}$ and $\mathcal{H}$ be Lie groupoids.
A {\it generalized morphism} $f : \mathcal{G} \rightarrow \mathcal{H}$ is a pair $f = (e_f, f')$ of a weak equivalence $e_f : \cdot \rightarrow \mathcal{G}$ and a strict morphism $f' : \cdot \rightarrow \mathcal{H}$.
(See the following.)
\[
f : \mathcal{G} \overset{e_f}{\leftarrow} \cdot \overset{f'}{\rightarrow} \mathcal{H}
\]
We call the above diagram a {\it right roof}.
Then, we write $f = f'e_f^{-1}$.
We call this notation a {\it right fraction}.
A generalized morphism $f'e_f^{-1}$ is a {\it weak equivalence} if the morphism $f'$ is a weak equivalence.
A Lie groupoid $\mathcal{G}$ is weak equivalnt to a Lie groupoid $\mathcal{H}$ if there exists a weak equivalence $\mathcal{G} \rightarrow \mathcal{H}$.
\end{defi}
\begin{ex} \label{leafspex}
Let $E$ be a $n$-manifold, and let $\mathcal{F}$ be a $p$-dimensional foliation on $E$.
Let $S$ and $T$ be global sections of $\mathcal{F}$, and let $\mathcal{G}$ (resp. $\mathcal{H}$) be the holonomy groupoid on $S$ (resp. $T$).
Then, the Lie groupoid $\mathcal{G}$ is weak equivalnt to the Lie groupoid $\mathcal{H}$.
A representative of the weak equivalence class of the holonomy groupoids is regarded as the leaf space of the foliation $\mathcal{F}$.
\end{ex}
\begin{rem}
Strictly speaking, a {\it generalized morphism} is a suitable equivalence class of roofs.
In fact, the class $W$ of weak equivalences admits caluculous of right fractions.
A generalized morphism is a morphism of a localization ${\bf LieGrpd}[W^{-1}]$.
(See \cite{gabriel2012calculus}.)
Remark, ${\bf LieGrpd}[W^{-1}]$ is a $2$-category.
\end{rem}
\vspace{10pt}
We introduce the {\it weak pull-back}. \\
Let $f : \mathcal{G} \rightarrow \mathcal{K}$ and $g : \mathcal{H} \rightarrow \mathcal{K}$ be strict morphisms between Lie groupoids. 
Let $P_0 (= \mathcal{G}_0 \times_{\mathcal{K}_0} \mathcal{K}_1 \times_{\mathcal{K}_0} \mathcal{H}_0) = \{ (x,\sigma,y) \in \mathcal{G}_0 \times \mathcal{K}_1 \times \mathcal{H}_0 \, | \, f(x) = s(\sigma) , t(\sigma) = g(y) \}$.
The set $P_0$ is {\it not} always a manifold. 
But, if $f \times g : \mathcal{G}_0 \times \mathcal{H}_0 \rightarrow \mathcal{K}_0 \times \mathcal{K}_0$ and $s \times t : \mathcal{K}_1 \rightarrow \mathcal{K}_0 \times \mathcal{K}_0$ intersect transversally, then the set $P_0$ is a manifold. 
In this case, the set $P_1 (= \mathcal{G}_1 \times_{\mathcal{K}_0} \mathcal{K}_1 \times_{\mathcal{K}_0} \mathcal{H}_1) = \{ (\gamma,\sigma,\delta) \in \mathcal{G}_0 \times \mathcal{K}_1 \times \mathcal{H}_0 \, | \, f(s(\gamma)) = s(\sigma) , t(\sigma) = g(s(\delta)) \}$ is also a manifold. 
Moreover, $\mathcal{P} = (\mathcal{P}_0 , \mathcal{P}_1)$ is a Lie groupoid.
\begin{rem}
The above condition involves the cases that either $f_0$ or $g_0$ is a submersion, and that either $f$ or $g$ is an essentially surjective. Because of $\mathcal{G}_0 \times_{\mathcal{K}_0} \mathcal{K}_1 \times_{\mathcal{K}_0} \mathcal{H}_0 \cong \mathcal{G}_0 \times_{\mathcal{K}_0} (\mathcal{K}_1 \times_{\mathcal{K}_0} \mathcal{H}_0) \cong (\mathcal{G}_0 \times_{\mathcal{K}_0} \mathcal{K}_1) \times_{\mathcal{K}_0} \mathcal{H}_0$.
\end{rem}
\begin{defi} \label{wpl}
The above Lie groupoid $\mathcal{P}$ is the {\it weak pull-back} of $f$ and $g$. 
And the following diagram is called {\it weak pull-back diagram}.
	\[\xymatrix{
		\mathcal{P} \ar[r]^-{pr_3} \ar[d]_-{pr_1}
		& \mathcal{H} \ar[d]^-{g}
	\\
		\mathcal{G} \ar[r]^-{f}
		& \mathcal{K}
	}\] 
\end{defi}
\begin{rem}
 Weak pull-back diagrams are {\it not} commutative but the above $f \circ pr_1$ and $g \circ pr_3$ are isomorphic.
\end{rem}
\begin{prop}[\cite{moerdijk2003introduction}]
In above case, if $f$ is a weak equivalence, then $pr_3 : \mathcal{P} \rightarrow \mathcal{H}$ is also weak equvalence.
\end{prop}
Let $g : \mathcal{G} \dashrightarrow \mathcal{H} , f : \mathcal{H} \dashrightarrow \mathcal{K}$ be generalized maps. 
And, take the weak pull-back $\mathcal{P}$ of $e_f$ and $g'$. 
The diagram is the following:
	\[\xymatrix{
		\mathcal{G} \ar@{-->}[rr]^-{g}
		&& \mathcal{H} \ar@{-->}[rr]^-{f}
		&& \mathcal{K}
	\\
		& \mathcal{G}_g \ar[ul]^-{e_g}  \ar[ur]_-{g'}
		&& \mathcal{H}_f  \ar[ul]^-{e_f}  \ar[ur]_-{f'}
	\\
		&& \mathcal{P}  \ar[ul]^-{e}  \ar[ur]_-{g''}
	}\] 
We define $f \circ g = (e_g \circ e,\mathcal{P} ,f' \circ g'')$.
\begin{rem}
If $f , g$, and $h$ are strict morphisms, then strictly $(f \circ g) \circ h = f \circ (g \circ h)$. 
On the other hand, If $f , g$, and $h$ are generalized maps, then $(f \circ g) \circ h \cong f \circ (g \circ h)$.
\end{rem}
	\subsection{Orbifolds}
\begin{defi} \label{defproper}
A Lie groupoid $\mathcal{G}$ is {\it proper} if $s \times t : \mathcal{G}_1 \rightarrow \mathcal{G}_0 \times \mathcal{G}_0$ is a proper map.
(i.e. The pull-back of any compact subset of $\mathcal{G}_0 \times \mathcal{G}_0$ by $s \times t$ is compact.) \\
An {\it orbifold} is a proper \'{e}tale Lie groupoid. \\
An {\it orbifold map} is a generalized map between orbifolds. \\
An {\it isomorphism} between orbifolds is a weak equivalence between orbifolds.
\end{defi}
\begin{rem}
Suppose two Lie groupids $\mathcal{G}$ and $\mathcal{H}$ are weak equivalent.
Then, the Lie groupoid $\mathcal{G}$ is proper if and only if the Lie groupoid $\mathcal{H}$ is proper.
Howevevr, even if the Lie groupoid $\mathcal{G}$ is \'{e}tale, it doesn't mean that the Lie groupoid $\mathcal{H}$ is \'{e}tale.
See \cite{moerdijk2003introduction}.
\end{rem}
\begin{ex}
Let $M$ be a manifold, and let $G$ be Lie group.
Suppose that $G$ acts on $M$.
If $G$ is compact, then the action groupoid $G \ltimes M$ associated with the action $G ^\curvearrowright M$ is proper.
(cf. Example \ref{actlie})
If $G$ is compact and discrete (i.e. $G$ is a finite group), then the action groupoid $G \ltimes M$ associated with the action $G ^\curvearrowright M$ is an orbifold.
(cf. Example \ref{actdiscrete})
\end{ex}
Any orbifold is a pasting of the above example.
\begin{defi}
Let $\mathcal{G}$ be a Lie groupoid, and let $x \in \mathcal{G}_0$ be an object.
The {\it isotoropy group} $G_x$ at $x$ is an isomorphism group $\{ \sigma : x \rightarrow x \}$.
\end{defi}
\begin{prop} \label{orbchart}
Let $\mathcal{G}$ be an orbifold, and let $x \in \mathcal{G}_0$ be an object.
Then, there exists an open neighborhood $U$ of the point $x$ such that the restriction $\mathcal{G}_U$ of the orbifold $\mathcal{G}$ on the open set $U$ is isomorphic to the action groupoid $G_x \ltimes U$.
Where the group $G_x$ is the isotoropy group at $x$.
\end{prop}
\begin{proof}
See Proposition 5.30 in \cite{moerdijk2003introduction}.
\end{proof}
Finally, we discuss manifolds.
\begin{defi}
A Lie groupoid $\mathcal{G}$ is {\it free} if $s \times t : \mathcal{G}_1 \rightarrow \mathcal{G}_0 \times \mathcal{G}_0$ is injective.
\end{defi}
\begin{rem}
A Lie groupoid is free if and only if any isotoropy group is trivial. \\
Suppose two Lie groupids $\mathcal{G}$ and $\mathcal{H}$ are weak equivalent.
Then, the Lie groupoid $\mathcal{G}$ is free if and only if the Lie groupoid $\mathcal{H}$ is free.
\end{rem}
\begin{prop}
An orbifold $\mathcal{G}$ is isomorphic to the trivial groupoid of a manifold $M$ if and only if the orbifold $\mathcal{G}$ is free.
\end{prop}
\begin{proof}
Suppose that an orbifold $\mathcal{G}$ is isomorphic to the trivial groupoid of a manifold $M$.
The trivial groupoid of the manifold $M$ is free.
Then, the orbifold $\mathcal{G}$ is free. \\
Conversely, suppose an orbifold $\mathcal{G}$ is free.
(Then, any isotopory group is trivial.)
Let $M$ be a quotient set $\mathcal{G}_0 / \mathcal{G}_1$.
There exists a manifold structure on $M$ such that the quotient map $\mathcal{G}_0 \rightarrow M$ is a surjective local diffeomorphism because of Proposition \ref{orbchart}.
Then, the quotient map $\mathcal{G}_0 \rightarrow M$ defines a weak equivalence from the orbifold $\mathcal{G}$ to the trivial groupoid of the manifold $M$.
\end{proof}
That's why we often refer to a free orbifold as a {\it manifold} as well.


\bibliography{engel, contact, symplectic, category, homotopy, foliation, Lie_groupoid}
\bibliographystyle{plain}


\end{document}